\documentclass[12pt]{amsart}
\usepackage{fourier}
\usepackage[T1]{fontenc}
\usepackage{graphics}
\usepackage{amsfonts,amssymb,color}
\usepackage[mathscr]{eucal}
\usepackage{amsmath, amsthm}
\usepackage{mathrsfs}
\usepackage{amsbsy}
\usepackage{dsfont}
\usepackage{bbm}
\usepackage{wasysym}
\usepackage{stmaryrd}
\usepackage{url}

\input xypic
\xyoption{all}

\makeindex
\makeglossary

\begin{document}
\baselineskip = 16pt

\newcommand \ZZ {{\mathbb Z}}
\newcommand \NN {{\mathbb N}}
\newcommand \RR {{\mathbb R}}
\newcommand \PR {{\mathbb P}}
\newcommand \AF {{\mathbb A}}
\newcommand \GG {{\mathbb G}}
\newcommand \QQ {{\mathbb Q}}
\newcommand \CC {{\mathbb C}}
\newcommand \bcA {{\mathscr A}}
\newcommand \bcC {{\mathscr C}}
\newcommand \bcD {{\mathscr D}}
\newcommand \bcF {{\mathscr F}}
\newcommand \bcG {{\mathscr G}}
\newcommand \bcH {{\mathscr H}}
\newcommand \bcM {{\mathscr M}}
\newcommand \bcJ {{\mathscr J}}
\newcommand \bcL {{\mathscr L}}
\newcommand \bcO {{\mathscr O}}
\newcommand \bcP {{\mathscr P}}
\newcommand \bcQ {{\mathscr Q}}
\newcommand \bcR {{\mathscr R}}
\newcommand \bcS {{\mathscr S}}
\newcommand \bcV {{\mathscr V}}
\newcommand \bcU {{\mathscr U}}
\newcommand \bcW {{\mathscr W}}
\newcommand \bcX {{\mathscr X}}
\newcommand \bcY {{\mathscr Y}}
\newcommand \bcZ {{\mathscr Z}}
\newcommand \goa {{\mathfrak a}}
\newcommand \gob {{\mathfrak b}}
\newcommand \goc {{\mathfrak c}}
\newcommand \gom {{\mathfrak m}}
\newcommand \gon {{\mathfrak n}}
\newcommand \gop {{\mathfrak p}}
\newcommand \goq {{\mathfrak q}}
\newcommand \goQ {{\mathfrak Q}}
\newcommand \goP {{\mathfrak P}}
\newcommand \goM {{\mathfrak M}}
\newcommand \goN {{\mathfrak N}}
\newcommand \uno {{\mathbbm 1}}
\newcommand \Le {{\mathbbm L}}
\newcommand \Spec {{\rm {Spec}}}
\newcommand \Gr {{\rm {Gr}}}
\newcommand \Pic {{\rm {Pic}}}
\newcommand \Jac {{{J}}}
\newcommand \Alb {{\rm {Alb}}}
\newcommand \Corr {{Corr}}
\newcommand \Chow {{\mathscr C}}
\newcommand \Sym {{\rm {Sym}}}
\newcommand \Prym {{\rm {Prym}}}
\newcommand \cha {{\rm {char}}}
\newcommand \eff {{\rm {eff}}}
\newcommand \tr {{\rm {tr}}}
\newcommand \Tr {{\rm {Tr}}}
\newcommand \pr {{\rm {pr}}}
\newcommand \ev {{\it {ev}}}
\newcommand \cl {{\rm {cl}}}
\newcommand \interior {{\rm {Int}}}
\newcommand \sep {{\rm {sep}}}
\newcommand \td {{\rm {tdeg}}}
\newcommand \alg {{\rm {alg}}}
\newcommand \im {{\rm im}}
\newcommand \gr {{\rm {gr}}}
\newcommand \op {{\rm op}}
\newcommand \Hom {{\rm Hom}}
\newcommand \Hilb {{\rm Hilb}}
\newcommand \Sch {{\mathscr S\! }{\it ch}}
\newcommand \cHilb {{\mathscr H\! }{\it ilb}}
\newcommand \cHom {{\mathscr H\! }{\it om}}
\newcommand \colim {{{\rm colim}\, }} 
\newcommand \End {{\rm {End}}}
\newcommand \coker {{\rm {coker}}}
\newcommand \id {{\rm {id}}}
\newcommand \van {{\rm {van}}}
\newcommand \spc {{\rm {sp}}}
\newcommand \Ob {{\rm Ob}}
\newcommand \Aut {{\rm Aut}}
\newcommand \cor {{\rm {cor}}}
\newcommand \Cor {{\it {Corr}}}
\newcommand \res {{\rm {res}}}
\newcommand \red {{\rm{red}}}
\newcommand \Gal {{\rm {Gal}}}
\newcommand \PGL {{\rm {PGL}}}
\newcommand \Bl {{\rm {Bl}}}
\newcommand \Sing {{\rm {Sing}}}
\newcommand \spn {{\rm {span}}}
\newcommand \Nm {{\rm {Nm}}}
\newcommand \inv {{\rm {inv}}}
\newcommand \codim {{\rm {codim}}}
\newcommand \Div{{\rm{Div}}}
\newcommand \CH{{\rm{CH}}}
\newcommand \sg {{\Sigma }}
\newcommand \DM {{\sf DM}}
\newcommand \Gm {{{\mathbb G}_{\rm m}}}
\newcommand \tame {\rm {tame }}
\newcommand \znak {{\natural }}
\newcommand \lra {\longrightarrow}
\newcommand \hra {\hookrightarrow}
\newcommand \rra {\rightrightarrows}
\newcommand \ord {{\rm {ord}}}
\newcommand \Rat {{\mathscr Rat}}
\newcommand \rd {{\rm {red}}}
\newcommand \bSpec {{\bf {Spec}}}
\newcommand \Proj {{\rm {Proj}}}
\newcommand \pdiv {{\rm {div}}}
\newcommand \wt {\widetilde }
\newcommand \ac {\acute }
\newcommand \ch {\check }
\newcommand \ol {\overline }
\newcommand \Th {\Theta}
\newcommand \cAb {{\mathscr A\! }{\it b}}

\newenvironment{pf}{\par\noindent{\em Proof}.}{\hfill\framebox(6,6)
\par\medskip}

\newtheorem{theorem}[subsection]{Theorem}
\newtheorem{conjecture}[subsection]{Conjecture}
\newtheorem{proposition}[subsection]{Proposition}
\newtheorem{lemma}[subsection]{Lemma}
\newtheorem{remark}[subsection]{Remark}
\newtheorem{remarks}[subsection]{Remarks}
\newtheorem{definition}[subsection]{Definition}
\newtheorem{corollary}[subsection]{Corollary}
\newtheorem{example}[subsection]{Example}
\newtheorem{examples}[subsection]{examples}

\title{Mumford-Roitman argument on families}
\author{Kalyan Banerjee}

\address{Indian Institute of Science Education and Research, Mohali, India}

\email{kalyanb@iisermohali.ac.in}

\begin{abstract}
The aim of this note is to show that the properties like motivic decomposability, Chow theoretic decomposition of the diagonal etc. happens for the special member of a family if it happens for a general member of the family.
\end{abstract}

\maketitle

\section{Introduction}
In the paper \cite{Voi}, C.Voisin proves that if one has a family of smooth projective varieties degenerating into a projective variety having only ordinary double point singularities, and we know that the general member of this family has decomposable diagonal then one has the decomposition of diagonal for the special singular fiber (rather for the desingularization of it). This is the starting point of this paper. We study relations in the Chow groups, that is we have a family of projective varieties over a projective base. Suppose that we have a property P (like Chow theoretic decomposition of the diagonal, Motivic decomposability etc.), for the general member of the family, then can it be deduced for the special fiber of the same family. Conversely suppose that the special fiber of the family does not satisfy the property P, then can we say that the general fiber does not satisfy the property P.

To study this question we need to study the relation between Chow schemes parametrizing algebraic cycles of a fixed dimension and fixed degree, and the natural map from the Chow scheme to the Chow group of codimension $p$ cycles. In \cite{R}, it has been proved that the natural map from the Symmetric power of a smooth projective variety to its Chow group zero cycles, has the fibers equal to a countable union of Zariski closed subsets of the symmetric power. This is what was predicted by Mumford in \cite{M}. Due to the theory of Suslin-Voevodsky \cite{SV}, we know that there exists a natural map from the $k$-points of the Chow scheme of codimension $p$-cycles to the Chow group of codimension $p$ cycles on $X$ (here $k$ is the ground field over which $X$ is defined). In \cite{BG}, it has been proved that the fibers of this natural map are equal to a countable union of Zariski closed subsets in the Chow scheme.

The technique is to understand, what does it mean to say that two codimension $p$ cycles are rationally equivalent in terms of the Chow scheme. In terms of Chow schemes it means that some shift of the two cycles which are rationally equivalent are rationally connected on the Chow scheme. This equivalent definition of rational equivalence is the main point in proving the above result, that the fiber of the map from the Chow scheme to the Chow group is a countable union of Zariski closed subsets of the Chow scheme.

So it would be nice to understand what happens when we consider a family of projective varieties, and consider the pairs of points on the base and cycles belonging to the Chow schemes of the fiber over that point, which are rationally equivalent to zero. Then we prove that this subset in the corresponding product is a countable union of Zariski closed subsets of the product. The technique used to prove such a theorem is to follow the Roitman's technique present in \cite{R}, proving the fibers of the natural map from the symmetric powers of a smooth projective variety to the Chow group, is a countable union of Zariski closed subsets of the symmetric power.

\medskip

\textit{
Let us consider a family of smooth, projective varieties $X\to B$. Let us consider the tuples $(b,z_b)$ such that $z_b$ is supported on $C^p_{d,d}(X_b)$ and is rationally equivalent to zero on $X_b$, where $C^p_{d,d}(X)$ denote the Chow scheme of $X$ parametrizing degree $d$ codimension $p$ cycles. Then we prove that this set inside $B\times C^p_{d,d}(X)$ is a countable union of Zariski closed subset of  $B\times C^p_{d,d}(X)$ where $C^p_{d,d}(X)$ is the family of Chow schemes over $B$.}

\bigskip

This has following consequences. We can consider "special relations" on the Chow group of the fibers of the given family, for example, Chow theoretic decomposition of the diagonal, motivic decomposability etc. They do fit into the above framework, that is we consider all points in the base such that the fiber over that satisfies the Chow theoretic decomposition of the diagonal. Then it follows that the collection of all such points in the base is a countable union of Zariski closed subsets in the base.

Another consequence of the above result is related to the non-rationality of a cubic fourfold. We start with a rational cubic fourfold and it is expected that the Fano variety of lines on the cubic is birational to the Hilbert scheme of two points on a fixed K3 surface. From this it has been deduced in \cite{B}, that there exists a surface or a smooth projective curve inside the Fano variety of lines on the cubic such that the push-forward induced at the level of zero cycles has torsion kernel. So we want to understand the collection of all cubic such that there exists a curve or a surface in the Fano variety such that the kernel of the induced push-forward is torsion. We prove that this collection is a countable union of Zariski closed subsets in the parameter space of smooth cubic fourfolds in $\PR^5$.

\medskip

\textit{Let $X$ be a cubic fourfold in $\PR^5$ such that there exists a smooth projective surface $S$ in $F(X)$, the Fano variety of lines, the kernel of the Gysin push-forward induced by the closed embedding of $S$ into $F(X)$ is torsion. We prove that the collection of such $X$ in the parameter space of cubic fourfolds is a countable union of Zariski closed subsets of the parameter space.}

\medskip

So if the expected result is true that the rationality of the cubic implies that the Fano variety is birational to the Hilbert scheme of two points on a fixed K3 surface, then the above result tells us that to prove the non-rationality of a very general cubic fourfold, we need find one cubic such  that there does not exist any curve or surface inside the Fano variety of lines of the corresponding cubic, such that the push-forward induced at the level of zero cycle has torsion kernel.

Another application of the above result is that the Bloch's conjecture for the surfaces of general type with $p_g=0$. Suppose that we have a family of surfaces of general type, degenerating into a special surface. Suppose the generic member of the family satisfies Bloch's conjecture, then its diagonal is balanced in the sense of \cite{BS}. Then our theorem tells us that the diagonal of the special fiber is balanced. It means that the Bloch's conjecture holds true for the desingularization of the special fiber. So this tells us that in particular a family of surfaces of general type with $p_g=0$ cannot be degenerated into a K3 surface, or a surface with geometric genus greater than zero. So it means that the special fiber must be of geometric genus zero and for them the Bloch conjecture must hold by our theorem. It is known \cite{Gul},\cite{Gul2}, that Bloch's conjecture is true for numerical Godeaux surfaces with an involution, some numerical Campedelli surfaces and the surface of Craighero and Gattazzo. Such degenerations of numerical Godeaux surfaces have been studied in \cite{FPR}, and the de-singularizations of the degenerations will satisfy Bloch's conjecture.

{\small \textbf{Acknowledgements:} The author would like to thank the hospitality of IISER-Mohali, for hosting this project. The author is indebted to Vladimir Guletskii for many useful conversations relevant to the theme of the paper. The author is thankful to S.Rollenske for pointing out the result about degenerations of numerical Godeaux surfaces in \cite{FPR}.}

\section{The Mumford-Roitman argument on families}

Let $k$ be an algebraically closed field of characteristic $0$. Let $C^p_{d}(X)$ denote the relative Chow scheme of $X$ over $B$, where $X\to B$ is a flat family of projective varieties. This scheme represents the functor associating the relative cycles on the fiber product $X\times_ B Y$, to each scheme $Y$ over $B$, in  the sense of Suslin-Voevodsky [please see \cite{SV}]. Using the functoriality we can prove that the Chow scheme $C^p_d(X)$ is the family of Chow schemes $C^p_d(X_b)$ for $b$ varying in $B$. Then our aim is to prove the following theorem.

Here $C^p_{d_1,\cdots,d_n}(X)$ denote the  product $\prod_i C^p_{d_i}(X)$. Same for symmetric powers of $X$.

\begin{theorem}
Let us consider a family of smooth, projective varieties $X\to B$. Let us consider the tuples $(b,z_b)$ such that $z_b$ is supported on $C^p_{d,d}(X_b)$ and is rationally equivalent to zero on $X_b$, where $C^p_{d,d}(X)$ denote the Chow scheme of $X$ parametrizing degree $d$ codimension $p$ cycles. Then we prove that this set inside $B\times C^p_{d,d}(X)$ is a countable union of Zariski closed subset of  $B\times C^p_{d,d}(X)$ where $C^p_{d,d}(X)$ is the family of Chow schemes over $B$.
\end{theorem}

\begin{proof}
So let us consider the relation that $z_b=z_b^+-z_b^-$ is rationally equivalent to zero. That means that there exists $f:\PR^1\to C^p_{d,d}(X_b)$, such that
$$f(0)=z_b^{+}+\gamma,f(\infty)=z_b^{-}+\gamma\;.$$
In other words we have the following map $\ev:Hom^v(\PR^1_k,C^p_{d}(X))$, given by $f\mapsto (f(0),f(\infty))$ and image of $f$ is contained in $C^p_{d,d}(X_b)$.  Consider the subscheme $U_{v,d}(X)$ of $B\times Hom^v(\PR^1_k,C^p_{d}(X))$ consisting of pairs $(b,f)$ such that image of $f$ is contained in $C^p_{d}(X_b)$ (such a universal family exists, for example see theorem 1.4 in \cite{Ko}). That gives us the morphism $(b,f)\mapsto (b,f(0),f(\infty))$, from $U_{v,d}(X)$ to $B\times C^p_{d,d}(X)$. Consider the closed subscheme $\bcV_{d,d}$ of $B\times C^p_{d,d}(X)$ given by $(b,z_1,z_2)$, such that $(z_1,z_2)$ belongs to $C^p_{d,d}(X_b)$. Then consider the map from $\bcV_{d,u,d,u}$ to $\bcV_{d+u,d+u}$ given by
$$(A,C,B,D)\mapsto (A+C,C,B+D,D)\;.$$
Then we can write the fiber product $\bcV$ of $U_{v,d}(X)$ and $\bcV_{d,u,d,u}$ over $\bcV_{d+u,d+u}$. If we consider the projection from $\bcV$ to $B\times C^p_{d,d}(X)$, then we get that $A$ and $B$ are supported and rationally equivalent on $X_b$. Conversely if $A,B$ are supported and rationally equivalent on $X_b$, then we have $f:\PR^1\to C^p_{d+u,u,d+u,u}(X_b)$ of some degree $v$ such that
$$f(0)=(A+C,C)\; f(\infty)=(B+D,D)\;,$$
where $C,D$ are supported on $X_b$. This analysis says that the image of the projection from $\bcV$ to $B\times C^p_{d,d}(X)$ , is a quasi-projective subscheme $W_{d}^{u,v}$ consisting of tuples $(b,A,B)$ such that $A,B$ are supported on $X_b$ and there exists $f:\PR^1_k\to C^p_{d+u,u}(X_b)$ such that $f(0)=(A+C,C)$ and $f(\infty)=(B+D,D)$, where $f$ is of degree $v$ and $C,D$ are supported on $X_b$ and they are co-dimension $p$ and degree $u$ cycles. So it means that $W_d=\cup_{u,v}W_d^{u,v}$. Now we prove that the Zariski closure of $W_d^{u,v}$ is in $W_d$ for each $u,v$.

For that we prove the following,

$$W_d^{u,v}=pr_{1,2}(\wt{s}^{-1}(W^{0,v}_{d+u}\times W^{0,v}_u))$$
where
$$\wt{s}(b,A,B,C,D)=(b,A+C,B+D,C,D)$$
from $B\times C^p_{d,d,u,u}(X)$ to $B\times C^p_{d+u,d+u,u,u}(X)\;.$
Let $(b,A,B,C,D)$ be such that its image under $\wt{s}$ is in $W^{0,v}_{d+u}\times W^{0,v}_u$. It means that there exists an element $(b,g)$ in $B\times\Hom^v(\PR^1_k,C^p_{d+u})$ and another $(b,h)$ in $\Hom^v(\PR^1_k,C^p_{u}(X))$ such that $g(0)=A+C,g(\infty)=B+D$ and $h(0)=C,h(\infty)=D$ and the image of $g,h$ are contained in $C^p_{d+u}(X_b)$, $C^p_u(X_b)$ respectively. Let us consider $f=g\times h$, then $f$ belong to $\Hom^v(\PR^1_k,C^p_{d+u,u}(X))$, such that the image of $f$ is contained in $C^p_{d+u,u}(X_b)$ with
$$f(0)=(A+C,C),(f(\infty))=(B+D,D)\;.$$
It means that $(b,A,B)$ belong to $W^d_{u,v}$.

On the other hand suppose that $(b,A,B)$ belongs to $W^d_{u,v}$. Then there exists $f$ in $\Hom^v(\PR^1_k,C^p_{d+u,u}(X_b))$ such that
$$f(0)=(A+C,C),f(\infty)=(B+D,D)\;,$$
and image of $f$ is contained in the Chow scheme of the fiber $X_b$.
Compose $f$ with the projections to $C^p_{d+u}(X_b)$ and to $C^p_{u}(X_b)$, then we have $g$ in $\Hom^v(\PR^1_k,C^p_{d+u}(X))$ and $h\in\Hom^v(\PR^1_k,C^p_{u}(X))$, such that
$$g(0)=A+C,g(\infty)=B+D$$
and
$$h(0)=C,h(\infty)=D\;,$$
and we have that the image of $g,h$ are contained in the respective Chow schemes of the fibers $X_b$.
Therefore we have that
$$W_d=pr_{1,2}(\wt{s}^{-1}(W_{d+u}\times W_u))\;.$$

Then we prove that the closure of $W_d^{0,v} $ is contained in $W_d$. Let $(b,A,B)$ be a closed point in the closure of ${W_d^{0,v}}$. Let $W$ be an irreducible component of ${W_d^{0,v}}$ whose closure contains $(b,A,B)$. Let $U$ be an affine neighborhood of $(b,A,B)$ such that $U\cap W$ is non-empty. Then there is an irreducible curve $C$ in $U$ passing through $(b,A,B)$. Let $\bar{C}$ be the Zariski closure of $C$ in $\bar{W}$. The map
$$e:U_{v,d}(X)\subset B\times \Hom^v(\PR^1_k,C^p_{d}(X))\to C^p_{d,d}(X)$$
given by
$$(b,f)\mapsto (b,f(0),f(\infty))$$
is regular and $W_d^{0,v}$ is its image. Let us choose a curve $T$ in $U_{v,d}(X)$ such that the closure of $e(T)$ is $\bar C$.  Consider the normalization $\wt{T}$ of the Zariski closure of $T$. Let $\wt{T_0}$ be the pre-image of $T$ in the normalization. Now the regular morphism $\wt{T_0}\to T\to \bar C$ extends to a regular morphism from $\wt{T}$ to $\bar C$. Now let $(b,f)$ be a pre-image of $(b,A,B)$. Then we have $f(0)=A;, f(0)=B$ and the image of $f$ is contained in $C^p_{d}(X_b)$ by definition of $U_{v,d}(X)$. Therefore $A,B$ are  rationally equivalent. This finishes the proof.
\end{proof}

\section{Application of the above result}

\subsection{Chow theoretic decomposition of the Diagonal}

In the paper by Voisin \cite{Voi}, an important stable birational invariant called Chow theoretic decomposition of the diagonal was introduced. Let $X$ be a smooth, projective variety, we say that the diagonal of $X$ admits Chow theoretic decomposition if the diagonal is rationally equivalent to $X\times x+Z$, where $Z$ supported on $D\times X$, where $D$ is a proper Zariski closed subset of $X$. Since all projective spaces admit Chow theoretic decomposition of diagonal, any variety which is stably rational, admits this property. So it is an invariant for stable rationality. In \cite{Voi}, it has been proved that if we have a flat family $X\to B$ such that the special fiber $X_0$ has atmost ordinary double point singularities, then the Chow theoretic decomposition of the general fiber implies that of the special fiber [therorem 1.1 in loc.cit.]. In particular if the special fiber $X_0$ does not admit of the Chow theoretic decomposition, then the general fiber does not admit of the Chow theoretic decomposition, hence is not stably rational.

This follows from previous theorem also that if the special fiber in a flat family does not admit of the Chow theoretic decomposition then the general fiber does not admit of the Chow theoretic decomposition. This is because the collection of points in the base, for which the corresponding fibers admit Chow theoretic decomposition is a countable union of Zariski closed subsets in the base. In particular the general fiber admits the decomposition means that it happens for all fibers (if the ground field is uncountable). Therefore we have that, if the special fiber does not admit a Chow theoretic decomposition then the general fiber does not admit such decomposition.

\subsection{Balancedness of the diagonal}

Similar argument holds true for Balancedness of the diagonal in family, meaning that if the diagonal is not balanced for the special fiber then it is not balanced for the general fiber.  The diagonal of a variety is balanced means that it is rationally equivalent to a sum of two cycles which are supported on $D\times X,X\times Z$, where $D,Z$ are proper Zariski closed subsets of $X$. For details please see \cite{BS},\cite{BL}. In particular if we have a family of surfaces of general type (rather a degeneration), then the Bloch's conjecture for the general surface implies the balanced ness of the diagonal for the general surface which would further imply it for the special surface appearing as a degeneration of a family of surfaces of general types.

\subsection{Motivic indecomposability}

In the paper by \cite{Gul}, the motivic indecomposability of the Chow motive of a smooth projective variety has been defined. The Chow motive of a smooth projective variety is said to be motivically decomposable if the diagonal is rationally equivalent to a sum of two codimension $d$ cycles on $X$ ($d$ is the dimension of $X$) each of which is non-torsion, non-balanced and not numerically trivial. So if we want to consider the motivic decomposability in a family then we have to consider the three conditions,

I) The diagonal is a sum of two correspondences

II) Each of the two correspondences are non-torsion

III) Each of the two correspondences are non-balanced

and one more that is each of them are not numerically trivial, but unfortunately we donot know how to encode numerical triviality in terms of Chow schemes.

So the collection points on the base for which the diagonal is a sum of two correspondences is a countable union of Zariski closed subsets of the base. The collection of points of the base for which the two correspondences are non-torsion and non-balanced is a countable intersection of Zariski open subsets in the base. So the points of the base for which motivic decomposability holds is a subset of a countable union of locally Zariski closed subsets of the base, hence it is a subset of a constructible set.

\subsection{Decomposition of the graph of a closed embedding of a hyperplane section}
Let us consider the Family of hyperplane sections of a smooth projective variety $X$ over a base $B$ (say for example that $B$ is a  projective line and the fibration is a Lefschetz pencil). Then we have the closed embedding of $X_t$ into $X$, denote it by $j_t$, we say that the graph of $j_t$ is decomposable, if as a cycle on $X_t\times X$, it is rationally equivalent to $X_t\times x+Z$, where $Z$ is supported on $X_t\times D$, $D$ proper, Zariski closed inside $X$. Then the theorem says that the decomposability of the graph of $j_t$ is a c-closed condition on the base. Meaning that the points in the base for which $j_t$ has decomposition, is a countable union of Zariski closed subsets in the base $B$. If the variety is a smooth projective surface, the by monodromy \cite{BG} it follows that for a general hyperplane section the graph does not have a decomposition, as the decomposition of the graph along with Chow moving lemma gives us that $j_{t*}=0$.

\section{Mumford-Roitman argument and non-torsion ness of the Gysin kernel}

In this section our aim is to prove the following.

\begin{theorem}
\label{theorem2}
Let $X$ be a cubic fourfold in $\PR^5$ such that there exists a smooth projective surface $S$ in $F(X)$, the Fano variety of lines, the kernel of the Gysin push-forward induced by the closed embedding of $S$ into $F(X)$ is torsion. We prove that the collection of such $X$ in the moduli of cubic fourfolds is a countable union of Zariski closed subsets of the parameter space.
\end{theorem}

\begin{proof}
Consider the family  of cubic fourfolds, denote the family by $\bcM$. Let us consider the the universal family of lines $\bcU$, given by pairs $(l,X)$, $l$ lying on $X$, in $\bcM\times G(1,\PR^5)$, where $G(1,\PR^5)$ is the Grassmannian of lines on $\PR^5$. Then $\bcU\to \bcM$ parametrizes the Fano variety of lines on the cubics in $\PR^5$. Then $\bcU$ is embedded into some $\PR^n$. Then consider the Chow vartiety $\bcC_{d}(\PR^n)$ of 2-cycles on $\PR^n$ and consider the subscheme of the product $\bcM\times \bcC_{d}(\PR^n) $ consisting of pairs $(S,m)$ such that $S\subset F_m$. Call this subscheme $\bcV$. Now the family $\bcU$ is a family over $\bcM$, so consider the relative symmetric power $\Sym^n (\bcU/\bcM)$ in the sense of Suslin-Voevodsky. So consider the tuples
$$(z,S,m)$$
such that $z$ is supported on $\Sym^n S$ and $S\subset F_m$. This is sub-variety inside $\Sym^n (\bcU/\bcM)\times \bcV$. Now consider the two conditions $j_*(z)$ is rationally equivalent to $0$ and $dz$ rationally equivalent to zero. By similar argument like  theorem 1 it follows that each of these conditions are c-closed. So the collection of $(z,S,m)$ such that $j_*(z)$ is rationally equivalent to zero  is a countable union of Zariski closed subsets in the ambient product variety. The collection of $(z,S,m)$, such that $z$ is torsion and supported on a symmetric power of $S$ is also a Zariski closed condition. So the complement of this countable union is a countable intersection of Zariski open sets in the product varietyt. So the collection of cubics for which for all surfaces in the Fano variety, there exists non-torsion elements in the Gysin kernel is a constructible condition. So the complement is a countable union of Zariski closed subsets in the moduli of cubics.

For the convenience let us give the proof parallel to theorem 1.

Let $z$ be a zero cycle supported on $S$, such that $j_*(z)=z^+-z^-$ is rationally equivalent to zero on $\CH_0(F_m)$.
That means that there exists $f:\PR^1\to \Sym^{d,d}(F_m)$, such that
$$f(0)=z^{+}+\gamma,f(\infty)=z^{-}+\gamma\;.$$
In other words we have the following map $\ev:Hom^v(\PR^1_k,\Sym^d(\bcU/\bcM))\to \Sym^{d,d}(\bcU/\bcM)$, given by $f\mapsto (f(0),f(\infty))$ and image of $f$ is contained in $\Sym^{d,d}(F_m)$.  Consider the subscheme $U_{v,d}(\bcU)$ of $ \bcV\times Hom^v(\PR^1_k,\Sym^d(\bcU/\bcM))$ consisting of pairs $(m,S,f)$ such that image of $f$ is contained in $\Sym^d(F_m)$, $S\subset F_m$ (such a universal family exists, for example see theorem 1.4 in \cite{Ko}). That gives us the morphism $(m,S,f)\mapsto (m,S,f(0),f(\infty))$, from $U_{v,d}(\bcU)$ to $ \bcV \times\Sym^{d,d}(\bcU/\bcM)$.

Consider the closed subscheme $\bcV_{d,d}$ of $ \Sym^{d,d}(\bcU/\bcM)\times \bcV$ given by $(z_1,z_2,S,m)$, such that $(z_1,z_2)$ belongs to $\Sym^{d,d}(S)$. Similarly we can construct $\bcV_{d,u,d,u}$. Then consider the map from $\bcV_{d,u,d,u}$ to $\bcV_{d+u,u,d+u,u}$ given by
$$(S,m,A,C,B,D)\mapsto (S,m,A+C,C,B+D,D)\;.$$

Then we can write the fiber product $\bcW$ of $U_{v,d}(\bcU)$ and $\bcV_{d,u,d,u}$ over $\bcV{d+u,u,d+u,u}(\bcU)$. If we consider the projection from $\bcW$ to $\bcV\times \Sym^{d,d}(\bcU/\bcM)$, then we get that $A$ and $B$ are supported on $S$ and $S\subset F_m$ and $A,B$ are rationally equivalent on $F_m$. Conversely if $A,B$ are supported and rationally equivalent on $F_m$ and supported on $S\subset F_m$, then we have $f:\PR^1\to \Sym^{d+u,u,d+u,u}(F_m)$ of some degree $v$ such that
$$f(0)=(A+C,C)\; f(\infty)=(B+D,D)\;,$$
where $C,D$ are supported on $F_m$ and we have that $(A,B)$ in the image of the projection of $\bcW\to \bcV\times \Sym^{d,d}(\bcU/\bcM)$. This analysis says that the image of the projection from $\bcW$ to $\bcV\times \Sym^{d,d}(\bcU/\bcM)$ , is a quasi-projective subscheme $W_{d}^{u,v}$ consisting of tuples $(m,A,B)$ such that $A,B$ are supported on $S\subset F_m$ and there exists $f:\PR^1_k\to \Sym^{d+u,u}(F_m)$ such that $f(0)=(A+C,C)$ and $f(\infty)=(B+D,D)$, where $f$ is of degree $v$ and $C,D$ are supported on $F_m$ and they are  degree $u$ cycles, relative zero cycles. So it means that $W_d=\cup_{u,v}W_d^{u,v}$. Now we prove that the Zariski closure of $W_d^{u,v}$ is in $W_d$ for each $u,v$.

For that we prove the following,

$$W_d^{u,v}=pr_{1,2}(\wt{s}^{-1}(W^{0,v}_{d+u}\times W^{0,v}_u))$$
where
$$\wt{s}(m,S,A,B,C,D)=(m,S,A+C,B+D,C,D)$$
from $\bcV\times \Sym^{d,d,u,u}(\bcU/\bcM)$ to $\bcV\times \Sym^{d+u,d+u,u,u}(\bcU/\bcM)\;.$
Let $(m,A,B,C,D)$ be such that its image under $\wt{s}$ is in $W^{0,v}_{d+u}\times W^{0,v}_u$. It means that there exists an element $(m,g)$ in $\bcV\times\Hom^v(\PR^1_k,\Sym^{d+u}(\bcU/\bcM))$ and another $(m,h)$ in $\Hom^v(\PR^1_k,\Sym^{u}(\bcU/\bcM))$ such that $g(0)=A+C,g(\infty)=B+D$ and $h(0)=C,h(\infty)=D$ and the image of $g,h$ are contained in $\Sym^{d+u}(F_m)$, $\Sym^u(F_m)$ respectively and $A+C,C,B+D,D$ are supported on $S\subset F_m$. Let us consider $f=g\times h$, then $f$ belong to $\Hom^v(\PR^1_k,\Sym^{d+u,u}(\bcU/\bcM))$, such that the image of $f$ is contained in $\Sym^{d+u,u}(F_m)$ with
$$f(0)=(A+C,C),(f(\infty))=(B+D,D)\;.$$
It means that $(m,S,A,B)$ belong to $W^d_{u,v}$.

On the other hand suppose that $(m,S,A,B)$ belongs to $W^d_{u,v}$. Then there exists $f$ in $\Hom^v(\PR^1_k,\Sym^{d+u,u}(F_m))$ such that
$$f(0)=(A+C,C),f(\infty)=(B+D,D)\;,$$
and image of $f$ is contained in the symmetric power of the fiber $F_m$.
Compose $f$ with the projections to $\Sym^{d+u}(F_m)$ and to $\Sym^u(F_m)$, then we have $g$ in $\Hom^v(\PR^1_k,\Sym^{d+u}(\bcU/\bcM))$ and $h\in\Hom^v(\PR^1_k,\Sym^{u}(\bcU/\bcM))$, such that
$$g(0)=A+C,g(\infty)=B+D$$
and
$$h(0)=C,h(\infty)=D\;,$$
and we have that the image of $g,h$ are contained in the respective symmetric powers  of the fibers $F_m$.
Therefore we have that
$$W_d=pr_{1,2}(\wt{s}^{-1}(W_{d+u}\times W_u))\;.$$

Then we prove that the closure of $W_d^{0,v} $ is contained in $W_d$. Let $(m,S,A,B)$ be a closed point in the closure of ${W_d^{0,v}}$. Let $W$ be an irreducible component of ${W_d^{0,v}}$ whose closure contains $(m,S,A,B)$. Let $U$ be an affine neighborhood of $(m,S,A,B)$ such that $U\cap W$ is non-empty. Then there is an irreducible curve $C$ in $U$ passing through $(m,S,A,B)$. Let $\bar{C}$ be the Zariski closure of $C$ in $\bar{W}$. The map
$$e:U_{v,d}(\bcU)\subset \bcV\times \Hom^v(\PR^1_k,\Sym^{d}(\bcU/\bcM))\to \Sym^{d,d}(\bcU/\bcM)$$
given by
$$(m,S,f)\mapsto (m,S,f(0),f(\infty))$$
is regular and $W_d^{0,v}$ is its image. Let us choose a curve $T$ in $U_{v,d}(\bcU)$ such that the closure of $e(T)$ is $\bar C$.  Consider the normalization $\wt{T}$ of the Zariski closure of $T$. Let $\wt{T_0}$ be the pre-image of $T$ in the normalization. Now the regular morphism $\wt{T_0}\to T\to \bar C$ extends to a regular morphism from $\wt{T}$ to $\bar C$. Now let $(m,S,f)$ be a pre-image of $(m,S,A,B)$. Then we have $f(0)=A;, f(\infty)=B$ and the image of $f$ is contained in $\Sym^{d}(F_m)$ by definition of $U_{v,d}(\bcU)$ and $S\subset F_m$ and $A,B$ are supported on $S$. Therefore $A,B$ are  rationally equivalent. This finishes the proof.

Similarly we can encode the condition of torsionness of a zero cycle supported on a smooth projective surface inside the Fano variety of lines on a cubic fourfold.

\end{proof}

\begin{remark}
This above theorem \ref{theorem2} implies that the collection of cubic fourfolds $X$ in $\PR^{55}$, for which there exists a smooth projective surface $S$ inside the Fano variety of lines $F(X)$, such the kernel of the push-forward induced by the embedding of $S$ into $F(X)$, is torsion, is a countable union of Zariski closed subsets in $\PR^{55}$. Now it is predicted that if the cubic is rational then it belongs to this countable union. So if the cubic does not belong to the countable union then it is irrational. So the only thing that we have to show for proving the irrationality of a very general cubic fourfold is that there exists a cubic fourfold for which, for all smooth projective surfaces in the Fano variety, the Gysin kernel contains a non-torsion point.
\end{remark}

\end{document}